\documentclass{amsart}
\usepackage[utf8]{inputenc}

\title{Escape from compact sets of normal curves in Carnot groups}

 \author[Le Donne]{Enrico Le Donne}
\address{E. Le Donne;
University of Fribourg, Chemin du Mus\'ee~23, 1700 Fribourg, Switzerland \& Department of Mathematics and Statistics \\
P.O. Box 35, FI-40014 \\
University of Jyv\"askyl\"a, Finland
\\
\href{enrico.ledonne@unifr.ch}{enrico.ledonne@unifr.ch}}

	\author[Paddeu]{Nicola Paddeu}
 \address{N. Paddeu;
 {University of Fribourg, Chemin du Mus\'ee~23, 1700 Fribourg, Switzerland  
 \\
 \href{nicola.paddeu@unifr.ch}{nicola.paddeu@unifr.ch}}
}

\date{6 April 2023}
\usepackage[english]{babel}
\usepackage[utf8]{inputenc}
\usepackage{libertine}
\usepackage{graphicx}
\usepackage{floatflt}
\usepackage{blindtext}
\usepackage{enumitem}
\usepackage{amsthm}
\usepackage{subfig}
\usepackage{listings}
\usepackage{listingsutf8}
\usepackage{amsmath}
\usepackage{framed}
\usepackage{minibox}
\usepackage{float}
\usepackage{wrapfig}
\usepackage{longtable}
\usepackage[strict]{changepage}
\usepackage{pgfplots}
\usepackage{nicefrac}
\usepackage{units}
\usepackage{etoolbox,tikz}
\usepackage{bbm}
\usepackage{hyperref}
\usepackage{faktor}
\usepackage{amsmath}
\usepackage{amssymb}

\usetikzlibrary{external}
\usetikzlibrary{cd}

\AtBeginEnvironment{tikzcd}{\tikzexternaldisable}
\AtEndEnvironment{tikzcd}{\tikzexternalenable}

\usepackage{amsmath}
\usepackage{amsfonts}

\usepackage{amsthm}
\usepackage {amsmath, amssymb}
\usepackage{pb-diagram}
\usepackage{tikz-cd}
\usepackage{mathtools}
\usepackage{amssymb}

\usetikzlibrary{matrix}
\pgfplotsset{width=11cm,compat=1.9}
\usepgfplotslibrary{external}
\newtheorem{teo}{Theorem}[section]
\newtheorem{defi}[teo]{Definition}
\newtheorem{prop}[teo]{Proposition}
\newtheorem{cor}[teo]{Corollary}
\newtheorem{lemma}[teo]{Lemma}
\newtheorem{oss}[teo]{Remark}
\newtheorem{Example}[teo]{Example}
\theoremstyle{plain}
\numberwithin{figure}{section}

\DeclareMathOperator{\Span}{Span}

\DeclareMathOperator{\Vecc}{Vec}

\DeclareMathOperator{\Ad}{Ad}

\DeclareMathOperator{\End}{End}
\DeclareMathOperator{\Lie}{Lie}
\DeclareMathOperator{\E}{E}

\newcommand{\df}{\mathrm{d}}
\newcommand{\rr}{\rightarrow}

\makeatletter
\DeclareRobustCommand*{\mfaktor}[3][]
{
   { \mathpalette{\mfaktor@impl@}{{#1}{#2}{#3}} }
}
\newcommand*{\mfaktor@impl@}[2]{\mfaktor@impl#1#2}
\newcommand*{\mfaktor@impl}[4]{
   \settoheight{\faktor@zaehlerhoehe}{\ensuremath{#1#2{#3}}}%
   \settoheight{\faktor@nennerhoehe}{\ensuremath{#1#2{#4}}}%
      \raisebox{-0.5\faktor@zaehlerhoehe}{\ensuremath{#1#2{#3}}}%
      \mkern-4mu\diagdown\mkern-5mu%
      \raisebox{0.5\faktor@nennerhoehe}{\ensuremath{#1#2{#4}}}%
}
\makeatother

\begin{document}
\maketitle
\begin{abstract}
In the setting of subFinsler Carnot groups, 
    we consider curves  that satisfy the normal equation coming from the Pontryagin Maximum Principle.
    We show that, unless it is constant, each such a curve leaves every compact set, quantitatively. Namely, the distance between the points at time 0 and time $t$ grows at least of the order of  $t^{1/s}$, where $s$ denotes the step of the Carnot group. In particular, in subFinsler Carnot groups there are no periodic normal geodesics.
\end{abstract}
\tableofcontents

\section{Introduction}


This paper originates from the following question: 
can subRiemannian manifolds have arbitrarily short geodesic loops?
Important subRiemannian manifolds are Carnot groups, which have dilation structures and therefore are self-similar. Hence, for these spaces, the question rephrases as:
do  geodesic loops exist in a Carnot group? 
We remark that not only geodesic loops do not exist in the Euclidean spaces, but also do not exist in normed vector spaces with strictly convex norm. Finite-dimensional normed vector spaces are exactly the Carnot groups of nilpotency step 1, equipped with geodesic distances. 
We refer to \cite{le2015metric} and \cite{le2018primer} for an introduction to Carnot groups and the fact that when equipped with their Carnot-Carath\'eodory distances they are exactly the metric spaces that are homogeneous, locally compact, geodesic, and admit dilations.\\
Some difficulties in Carnot groups are that the distance function may not be convex, that at every scale there exist pairs of points joined by more than one geodesic, and geodesics may not be globally length-minimizing. 
However, the biggest difficulty, as in other subRiemannian problems, is the presence of geodesics that are singular points of the endpoint map, which we call the {\em abnormal geodesics}.
By the Pontryagin Maximum Principle we know that if a geodesic is not abnormal, then it satisfies a geodesic equation called the {\em normal equation}. The curves satisfying the normal equation, called {\em normal curves} (or \emph{normal trajectories}), are more manageable. In this paper we focus on normal curves and we prove that, except for the constant ones, they cannot form loops and actually they  leave every compact set in a quantitative way.
The main result of this paper is the following statement.

\begin{teo}[Normal curves leave compact sets]
Let $G$ be a subFinsler Carnot group of step $s$. 
Fix a norm $N$ on the space of right-invariant co-vectors. Then, there exists a constant $\epsilon>0$ such that for every normal curve $\gamma:\mathbb{R}\rr G$ parametrized by arc-length there holds  
\begin{equation}
    d(\gamma(t),\gamma(t'))\geq \frac{\epsilon}{N(\lambda)^{\frac{1}{s}}}|t-t'|^{\frac{1}{s}}-1, \ \ \forall t,t'\in\mathbb{R},
    \label{e notinball}
\end{equation}
for every co-vector $\lambda$ associated to $\gamma$, see Definition~\ref{d right-inv-covector}.
\label{t fugacptquant}
\end{teo} 



As a consequence 
we rule out the presence of some geodesic loops.
\begin{cor}
In subFinsler Carnot groups normal loops are constant. 
\label{c noperiodic0}
\end{cor}
We stress that the above 
corollary has been know to be true in step 2 Carnot groups (by the complete integration of geodesics) and in jet spaces by the work of Bravo-Doddoli \cite{https://doi.org/10.48550/arxiv.2203.16178}.
Moreover, a similar question arose in geometric group theory of the large-scale geometry of nilpotent groups. In fact, because of the work of Hoda \cite{hodastronglyshortcutspaces}, it would be important to prove that Carnot groups cannot have isometric copies of the standard unit circle. The latter result is known for finite-dimensional normed spaces and was proven by Creutz, see \cite[Lemma 1.7]{creutz2021rigidity}.

\subsection{The basic case of loops in subRiemannian Carnot groups}
We give here a simpler proof of the non-existence of normal loops in subRiemannian Carnot groups. The same idea will be pushed to prove Corollary~\ref{c noperiodic0} and to show the quantitative result of Theorem~\ref{t fugacptquant}.

\begin{proof}[Proof of Corollary~\ref{c noperiodic0} for subRiemannian Carnot groups.]
Let $G$ be a subRiemannian Carnot group with left-invariant scalar product  $\langle\cdot,\cdot\rangle$.
Let $\gamma:[0,1]\rr G$ be 
a normal geodesic that makes a loop.
Without loss of generality we assume 
$\gamma(0)=\gamma(1)=1_G$. 

Since $\gamma$ is normal, by definition its control 
  $u$ satisfies the {\em normal equation}: for some 
  $\lambda\in (T_{\gamma(1)}G)^*$    and all controls $v$ we have
\begin{equation}
      \lambda(\df\End_u v)=\langle u ,v\rangle,
      \label{e pcor3}
\end{equation}
  where $\End$ is the end-point map from $1_G$, see for example \cite[Corollary 8.8]{agrachev2012introduction}. When $v=u$ we get
  \begin{equation}
      \lambda(\df\End_u u)=||u||_{L^2}^2.
      \tag{\ref{e pcor3}bis}
      \label{e 2bis}
  \end{equation}
 Now the idea is to consider the curve dilated by the Carnot dilations: for   $\tau>0$ let $\delta_\tau:G\to G $ be the dilation of factor $\tau$.
    On the one hand, we obviously have 
  $(\delta_\tau \circ \gamma )(1)=1_G$.
  On the other hand, the control of the curve $\delta_\tau \circ \gamma$ is $\tau u$.
  Hence
    we have that
 \begin{equation}
     \End(\tau u)=\delta_\tau(\End(u))=1_G,  \ \ \forall \tau>0.
     \label{e pcor1}
 \end{equation}
Equation (\ref{e pcor1}) implies that
\begin{equation}
    \df\End_u u=\frac{\df}{\df \epsilon} \End((1+\epsilon)u)\bigg|_{\epsilon=0}=0.
    \label{e pcor2}
\end{equation}
Consequently, equations (\ref{e pcor2}) and (\ref{e 2bis}) imply that $||u||_{L^2}=0$ and therefore $\gamma$ is the constant curve. 
\end{proof}

\subsection{The strategy of the proof}
We present here the key ideas that we shall use in the proof of Theorem~\ref{t fugacptquant}. We work in the context of subFinsler spaces, smooth manifolds equipped with a distribution, in which the distance between two points is given by the infimum of the lengths of paths tangent to the distribution joining the two points. The length is measured with respect to a continuously varying norm. \\
We start by writing the normal equation coming from the Pontryagin Maximum Principle in terms of the sub-differentials of the energy, see Proposition~\ref{p puntuale}. We use this formulation of the normal equation to get the analog of \eqref{e 2bis}, which expresses the energy of a normal control $u$ in terms of some co-vector $\lambda$ applied to $\df \End_u u$. 
We then use the properties of the one parameter subgroup of dilations $\delta$ to define the vector field
$$\Vec{\delta}(g):=\left.\frac{\df}{\df t}\delta_t(g)\right|_{t=1}, \forall g\in G.$$
The first key point is that we have $\Vec{\delta}(\End(u))=\df \End_uu$, see \eqref{e 1.6mainteo}. In particular we shall get 
\begin{equation}
     ||u||^2_{L^2}=\lambda(\Vec{\delta}(\End(u)). \label{e ideaproof}
\end{equation}
We show that $g \mapsto\lambda(\Vec{\delta}(g))$ is a finite sum of homogeneous functions of degree of homogeneity smaller than the step $s$, and therefore it can be bounded by $Cd(1,g)^s$ when $d(1,g)>1$, for some constant $C>0$. This bound together with \eqref{e ideaproof} will give \eqref{e notinball}.  
\subsection{Organization of the paper}
The paper is organized as follows: we dedicate Section~2 to a brief presentation of the notions of subFinsler geometry that we need in the paper. We recall the definitions of subFinsler manifolds, Carnot groups, and self-similar spaces; we define normal curves and we characterize them in terms of sub-differentials of the energy. In Section~3, we prove the main results of the paper. We also prove a slight generalization of Corollary~\ref{c noperiodic0} to self-similar spaces. Section~4 contains some examples: we show that length-$1$ normal geodesics can stay arbitrarily close to the starting point and that the exponent $\frac{1}{s}$ in Theorem~\ref{t fugacptquant} is the biggest possible. 

\textbf{Acknowledgments}. 
The authors
were partially supported by 
the Swiss National Science Foundation
(grant 200021-204501 `\emph{Regularity of sub-Riemannian geodesics and applications}')
and by the European Research Council  (ERC Starting Grant 713998 GeoMeG `\emph{Geometry of Metric Groups}').  
E.L.D was also partially supported by the Academy of Finland 
 (grant 322898
`\emph{Sub-Riemannian Geometry via Metric-geometry and Lie-group Theory}').

\section{Preliminaries}
\subsection{SubFinsler geometry and Carnot groups}
We start recalling some basic definitions and facts from subFinsler geometry following \cite{ledonne2010lecture} and \cite{antonelli2022lipschitz}.
\begin{defi}
Let $M$ be a smooth manifold. A \emph{distribution} on $M$ is a sub-bundle of the tangent bundle. We call a distribution \emph{bracket-generating} if, in some neighborhood of every point, the Lie algebra generated by vector fields tangent to the distribution contains a frame for the tangent bundle on that neighborhood. A \emph{subFinsler manifold} is a smooth manifold $M$, with a bracket-generating distrubution $\Delta$ and a continuously varying norm $||\cdot||$, equipped with the Carnot-Carathéodory distance
\begin{equation}
\begin{split}
        d_{cc}(x,y):=\inf\{\int_0^1||\dot{\gamma}(t)||\df t \ \big|\gamma:[0,1]\rr M \text{ absolutely continuous; }\\\dot{\gamma}(t)\in \Delta_{\gamma(t)},\text{ for a.e. } t\in[0,1]; \gamma(0)=x; \gamma(1)=y \}.
\end{split}
\label{e dcc}
\end{equation}
\end{defi}
The fact that $d_{cc}$ is a distance is guaranteed by Chow's Theorem (see for example \cite[Section 1]{gromov1996carnot}).
In this paper we focus mainly on Carnot groups: particular Lie groups with left-invariant distributions and norms.
\begin{defi}
A \emph{stratification} of a Lie algebra $\mathfrak{g}$ is a decomposition of $\mathfrak{g}$ as a direct sum $\mathfrak{g}=V_1\oplus...\oplus V_s$ with  
\begin{equation}
    [V_1,V_j]=V_{j+1}, \ \ \forall j\in\{1,...,s\},
\end{equation}
where $V_{s+1}:=\{0\}$. The sub-spaces $V_i$ are called \emph{strata} of the stratification. A Lie algebra equipped with a stratification is called \emph{stratified}.
A \emph{Carnot group} is a simply connected Lie group with stratified Lie algebra, with the left-invariant extension of the first stratum as distribution and a left-invariant norm, equipped with the Carnot-Carathéodory distance.
\end{defi}
Carnot groups are self-similar, in the sense that
there are natural automorphisms that act as homotheties.
\begin{defi}
Let $G$ be Carnot group with stratified Lie algebra $\mathfrak{g}=V_1\oplus...\oplus V_s$. The \emph{dilation} $\delta_\lambda:G\rr G$ of factor $\lambda$, with $\lambda\in \mathbb{R}$, is the Lie group automorphism defined setting
\begin{equation}
    (\delta_\lambda)_*(v)=\lambda^iv, \ \ ,\forall i\in\{1,...,s\}, \  \forall v\in V_i.
\end{equation}
\end{defi}
The presence of dilations  will be crucial in our proof of the non-existence of periodic normal geodesics, see  \eqref{formula_end_d}.
\subsection{Homogeneous spaces and self-similar distances}
In differential geometry, the term {\em homogeneous space} is referred to the quotient space of a Lie group modulo a closed subgroup, in order to still have a transitive action of the Lie group. However, in metric geometry and, more generally in analysis on metric spaces, the term homogeneous is referred to functions that gets multiplied by  some dilations of the space.
Every metric space admitting a dilation, also called  homothety, is said to be self-similar. 

In subFinsler geometry, self-similar spaces are well characterized. As differentiable manifold they have a homogeneous structure of a quotient space of a Carnot group modulo the action of a dilation-invariant subgroup via left-multiplication. 
However, the well-defined action on the right is not by isometries. Hence they are not isometrically homogeneous spaces. They are still called homogeneous because they admit dilations. To avoid this double use of the word homogeneity, we shall call them self-similar (subFinsler) spaces. By the work of Bella{\"\i}che (\cite{bellaiche1996tangent}, \cite{antonelli2022lipschitz}) we know that the metric tangents of (constant-rank) subFinsler manifolds are (constant-rank) self-similar spaces.

\begin{defi}\label{def_selfsimilar}
A \emph{self-similar subFinsler space} is a subFinsler manifold obtained as the quotient space of a (left-invariant) subFinsler Carnot group with respect to the left-action of a dilation-invariant subgroup, and it is equipped with the quotient distribution. Namely, 
assume $G$ is a Carnot group with distribution $\Delta$ and left-invariant norm $||\cdot||$ and $H<G$ is a closed dilation-invariant subgroup, for which $T_1H \cap \Delta_1=\{0\}$. 
On the quotient manifold $\mfaktor{H}{G}:=\{Hg : g\in G\}$ we define a subFinsler structure that makes the projection $\pi:G\rr \mfaktor{H}{G}$ a submetry: we take $\Delta_{\mfaktor{H}{G}}:=\pi_*\Delta$ as (constant-rank) distribution and we define the continuously varing norm on $\mfaktor{H}{G}$ setting for all $p\in \mfaktor{H}{G}$ and for all $v\in 
(\pi_*\Delta)_p\subseteq 
T_p(\mfaktor{H}{G})$
\begin{equation}
    (||v||_{\mfaktor{H}{G}})_p:=\inf\{||w||_q : q\in \pi^{-1}(p), w\in T_q\Delta, \df\pi_q(w)=v \}.
    \label{d defnormadown}
\end{equation}
\end{defi}
Being $H$ dilation invariant, the dilations of $G$ pass to the quotient and define a dilation on $\mfaktor{H}{G}$. Thus, on a self-similar space $\mfaktor{H}{G}$ we call $H$ the {\em origin}, since it is the only point fixed by dilations. We remand to \cite[Section 7]{bellaiche1996tangent} for 
a presentation of self-similar spaces.

\subsection{Pontryagin Maximum Principle and normal curves}
In the remaining part of this section we define the end-point map and we state the Pontryagin Maximum Principle. Afterwards, we present normal curves and we characterize them using sub-differentials of the energy.

\begin{defi}\emph{(Controls)}
Let $M$ be a subFinsler manifold with distribution $\Delta$, choose $X_1,...,X_n\in \Vecc(M)$ such that $\Span(\{X_1(p),...,X_n(p)\})=\Delta_p$ for every $p\in M$. 
The elements of $\Omega:=L^2([0,1],\mathbb{R}^n)$ are called {\em controls}. 
\end{defi}
Consider a subFinsler manifold $M$ and fix $X_1,...,X_n\in \Vecc(M)$ such that the distribution at every $p\in M$ coincides with $\Span(\{X_1(p),...,X_n(p)\})$. Let $\Omega$ be $L^2([0,1],\mathbb{R}^n) $.
To each control $u\in \Omega$ and each point $p\in M$ we associate the unique curve $\gamma_u:[0,1]\rr M$ solving
\begin{equation}
    \begin{cases}
    \gamma_u'(t)=\sum_{i=1}^n u_i(t)X_i(\gamma_u(t)),\text{ for a.e. } t\in[0,1];\\
    \gamma_u(0)=p.
    \end{cases}
    \label{e curvessociatedtocontrol}
\end{equation}
The existence and uniqueness of the solution of (\ref{e curvessociatedtocontrol}) is guaranteed by Carathéodory Existence and Uniqueness Theorem (\cite[Theorem 3.4]{o1997existence}, \cite[Theorem 1.3]{coddington1955theory}).
We shall say that $\gamma_u$ is a/the \emph{curve with control $u$}, or that $u$ is a \emph{control} of $\gamma_u$. Two controls could give the same curve, because the vector fields $X_j$'s might not be linearly independent at some point. Moreover, every curve of finite length admits a reparametrization with control in $\Omega$.
We can use Equation (\ref{e curvessociatedtocontrol}) to define the flow along a control. This will be useful for the computation of the differential of the end-point map.
\begin{defi}\emph{(Flow along a control)}
Let $M$ be a subFinsler manifold with distribution $\Delta$, let $X_1,...,X_n\in \Vecc(M)$ be such that $\Span(\{X_1(p),...,X_n(p)\})=\Delta_p$ for every $p\in M$ and define $\Omega:=L^2([0,1],\mathbb{R}^n) $. The \emph{flow along a control} $u\in \Omega$ is the map $\phi:[0,1]\times M\rr M$ defined as
\begin{equation}
    \phi^t(p):=\gamma_u(t), \ \forall t\in[0,1],
    \label{e def_of_flow}
\end{equation}
where $\gamma_u$ is the curve solving (\ref{e curvessociatedtocontrol}). 
For fixed $s,t\in [0,1]$, with $s<t$, we will write $\phi_s^t:M\rr M$ for the flow from $s$ to $t$:
\begin{equation}
    \phi_s^t:=\phi^{t-s}\circ (\phi^s)^{-1}.
\end{equation}
\label{d flow_along_control}
\end{defi}
\begin{Example}
Let $G$ be a subFinsler Carnot group with first stratum $V_1$. Let $X_1,...,X_n$ be a left-invariant frame of the first stratum. Using the frame $X_1,...,X_n$ identify $L^2([0,1],\mathbb{R}^n)$ with $L^2([0,1],V_1)$. For all $p\in G$ and $u\in L^2([0,1],V_1)$, equation \eqref{e curvessociatedtocontrol} rewrites as
\begin{equation}
    \begin{cases}
    \gamma_u'(t)=\df L_{\gamma_u(t)}u(t),\text{ for a.e. } t\in[0,1];\\
    \gamma_u(0)=p.
    \end{cases}
    \label{e curve_associated_to_control_Carnot}
\end{equation}
In particular, if $\phi:[0,1]\times G\rr G$ is the flow along a control $u\in L^2([0,1],V_1)$, see \eqref{e def_of_flow}, then it is trivial to check that for all $p\in G$ the curve $t\mapsto L_p(\phi^t(1_G))$ solves \eqref{e curve_associated_to_control_Carnot}. Consequently,
\begin{equation}
    \phi^t(p)=R_{\phi^t(1_G)}(p), \ \forall p\in G, \forall t\in[0,1].
    \label{e flow_is_right_translation}
\end{equation}

\end{Example}
\begin{defi}\emph{(End-point map)}
Let $M$ be a subFinsler manifold with distribution $\Delta$. Choose $X_1,....,X_n\in \Vecc(M)$ such that $\Span(\{X_1(p),...,X_n(p)\})=\Delta_p $. Denote $\Omega:=L^2([0,1],\mathbb{R}^n) $ and fix a point $p\in M$. The \emph{end-point map} $\End:\Omega\rr M$ associated to $X_1,...,X_n$ and $p$ is defined as 
 \begin{equation}
     \End(u):=\gamma_u(1),
 \end{equation}
 where $\gamma_u$ solves (\ref{e curvessociatedtocontrol}).
 In a Carnot group we always assume the point $p$ to be the identity element and $X_1,...,X_n$ to be a left-invariant frame of the first stratum.
\end{defi}

\begin{prop}\emph{(Differential of the end-point map, \cite[Proposition 8.5]{agrachev2012introduction}).}
Let $M$ be a subFinsler manifold with distribution $\Delta$, let $X_1,...,X_m\in \Vecc (M)$ such that $\Delta_p=\Span(\{X_1(p),...,X_m(p)\})$ for every $p\in M$. Fix $q\in M$. Set $\Omega:=L^2([0,1],\mathbb{R}^m)$ and denote with $\End:\Omega\rr M$ the end-point map associated to $X_1,...,X_m$ and $q$. Then $\End$ is smooth on $\Omega$ and for every $u,v\in \Omega$ we have
\begin{equation}
    \df \End_uv=\int_0^1 (\phi_t^1)_*\sum_{i=1}^m v_i(t)X_i(\gamma(t))\df t;
\end{equation}
where $\phi^t$ is the flow associated to the control $u$ and $\gamma(t):=\phi^t(q)$, see \eqref{e def_of_flow}.
\label{p differentialend}
\end{prop}
\begin{defi}\emph{(Energy)}
Let $M$ be a subFinsler manifold with continuously varying norm $||\cdot||:\Delta \rr \mathbb{R}$. The \emph{energy} at a point $p\in M$ is the function $\E_p:\Delta_p \rr \mathbb{R} $ defined as
\begin{equation*}
    \E_p:=\frac{1}{2}||\cdot||_p^2.
\end{equation*}
For an absolutely continuous curve $\gamma:[0,1]\rr M$ with $\dot{\gamma}\in \Delta$ we define its \emph{energy} as $\frac{1}{2}\int_0^1||\dot{\gamma}(t)||^2\df t $.
\end{defi}
The curves realizing the infimum in \eqref{e dcc} are called \emph{length-minimizing}. It is well-known that the same infimum is realized by minimizers of the energy, considering length-minimizing curves re-parametrized by constant speed.\\
We are now ready to state the well known Pontryagin Maximum Principle, which gives first-order necessary conditions for curves to be energy-minimizing.
\begin{teo}\emph{(PMP, \cite[Theorem 12.10]{agrachev2013control}).}
Let $M$ be a subFinsler manifold with distribution $\Delta$, let $X_1,...,X_m\in \Vecc (M)$ such that $\Delta_p=\Span\{X_1(p),...,X_m(p)\}$ for every $p\in M$. Let $\Omega$ be $L^2([0,1],\mathbb{R}^m) $.\\
For every $\nu\in\mathbb{R}$ and $v=(v_1,...,v_m)\in \mathbb{R}^m$;  define $h_{v,\nu}:T^*M\rr \mathbb{R}$ as
\begin{equation}
    h_{v,\nu}(\eta):=\left\langle \eta, \sum_{i=1}^m v_iX_i(\pi(\eta))  \right\rangle+ \nu \E_{\pi(\eta)}\left(\sum_{i=1}^m v_iX_i(\pi(\eta))\right), \ \ \forall \eta\in T^*M,
    \label{e defh}
\end{equation}
where $\pi:T^*M\rr M$ is the canonical projection.
If a curve $\gamma:[0,1]\rr M$ with control $u$ is energy-minimizing, then there exists $\nu\in\{-1,0\}$ and an absolutely continuous curve $\eta:[0,1]\rr T^*M$ such that $\eta(0)\in T^*_{\gamma(0)}M$ and
\begin{equation}
\dot{\eta}(t)=\Vec{h}_{u(t),\nu}(\eta(t)), \text{ for a.e. } t\in[0,1];
    \label{e pmp1}
\end{equation}
\begin{equation}
h_{u(t),\nu}(\eta(t))\geq h_{v,\nu}(\eta(t)), \ \ \forall v\in \mathbb{R}^m, \text{ for a.e. } t\in[0,1],
    \label{e pmp2}
\end{equation}
where $\Vec{h}_{u(t),\nu}(\eta(t))$ is the Hamiltonian vector field associated to $h_{u(t),\nu}$, see \cite[Section 11.5.2]{agrachev2013control}.
\label{p pmpgeneral}
\end{teo}
In this paper we will use a reformulation 
of Proposition~\ref{p pmpgeneral}. We first need to recall the definition of sub-differential.
\begin{defi}\emph{(Sub-differentials)}
Let $V$ be a vector space and $f:V\rr \mathbb{R}$. A \emph{sub-differential} of $f$ at $v\in V$ is a linear function $a:V\rr \mathbb{R}$ such that 
\begin{equation}
    a(u-v)\leq f(u)-f(v), \quad \forall u\in V.
\end{equation}
We will use $\partial_vf$ to denote the \emph{set of sub-differentials} of $f$ at $v$.
\end{defi}
It is an easy exercise to check that the set of sub-differentials of a convex function is always a non-empty closed convex set.
\begin{prop}\emph{(subFinsler PMP revised)}
Let $M$ be a subFinsler manifold with distribution $\Delta$, let $X_1,...,X_m\in \Vecc (M)$ such that $\Delta_p=\Span(\{X_1(p),...,X_m(p)\})$ for every $p\in M$. If a curve $\gamma:[0,1]\rr M$ with control $u$ is energy-minimizing then there exists $\lambda\in T^*_{\gamma(1)}M$ such that either
\begin{equation}
        \left(\phi_t^1\right)^*\lambda\in \partial_{\dot{\gamma}(t)}\E_{\gamma(t)}, \ \ \text{ for a.e. } t\in[0,1]
        \label{e normale}
\end{equation}
or
\begin{equation}
    \lambda\left(\left(\phi_t^1\right)_* X\right)
    =0, \ \ \forall t\in[0,1] , \ \forall  X\in \Delta_{\gamma(t)},
    \label{e abnormale}
\end{equation}
the function $\phi^t$ denoting the flow along the control $u$, see \eqref{e def_of_flow}.
\label{p puntuale}
\end{prop}
\begin{proof}
By Proposition~\ref{p pmpgeneral} if a curve $\gamma:[0,1]\rr M$ with control $u$ is energy-minimizing then there exist $\nu\in\{-1,0\}$ and an absolutely continuous curve $\eta:[0,1]\rr T^*M$ such that $\eta(0)\in T^*_{\gamma(0)}M$ and equations (\ref{e pmp1}) and (\ref{e pmp2}) hold. 
It is well know (see for example \cite[Section~12.2]{agrachev2013control}) that Equation (\ref{e pmp1}) rewrites for all $t\in[0,1]$ as
\begin{equation}
\begin{cases}
\pi(\eta(t))=\gamma(t);\\
\eta(t)=\left(\phi_t^1\right)^*\eta(1).
\end{cases}
\label{e supmp}
\end{equation}
Setting $\lambda:=\eta(1)$, by (\ref{e supmp}) and the definition of $h_{v,\nu}$ in (\ref{e defh}) we have that 
\begin{equation}\label{h_semplice}
    h_{v,\nu}(\eta(t))=\left\langle \left(\phi_t^1\right)^*\lambda, \sum_{i=1}^m v_iX_i(\gamma(t))  \right\rangle+ \nu \E_{\gamma(t)}\left(\sum_{i=1}^m v_iX_i(\gamma(t))\right) .
\end{equation}
When $\nu=-1$, 
Equation (\ref{e pmp2}) is equivalent to, for all $v\in\mathbb{R}^m$,
\begin{equation*}
\left\langle \left(\phi_t^1\right)^*\lambda, \sum_{i=1}^m v_iX_i(\gamma(t)) - \dot\gamma(t)  \right\rangle
\leq 
\E_{\gamma(t)}\left(\sum_{i=1}^m v_iX_i(\gamma(t))\right)  - \E_{\gamma(t)}\left(\dot\gamma(t)\right) .
 \end{equation*}
  When $v$ varies, the vector $\sum_{i=1}^m v_iX_i(\gamma(t))$ gives an arbitrary element of the domain of $\E_{\gamma(t)}$. Hence, we have  that Equation (\ref{e pmp2}) is
equivalent to (\ref{e normale}).
When $\nu=0$, 
from \eqref{h_semplice} Equation (\ref{e pmp2}) is
\begin{equation*}
\left\langle \left(\phi_t^1\right)^*\lambda, \sum_{i=1}^m v_iX_i(\gamma(t)) - \dot\gamma(t)  \right\rangle
\leq 0 .
 \end{equation*}
Noticing that 
we are applying $\left(\phi_t^1\right)^*\lambda$ to an arbitrary element of $
\Delta_{\gamma(t)} 
$, we finally deduce that in this second case
   Equation (\ref{e pmp2}) is
equivalent to   (\ref{e abnormale}).
\end{proof}

\begin{defi}\emph{(Normal and abnormal curves)}
In the setting of Proposition~\ref{p puntuale}, we say that a curve $\gamma:[0,1]\rr M$ is a \emph{normal curve} if there exists a co-vector $\lambda\in T^*_{\gamma(1)}M$ such that (\ref{e normale}) holds. We say that a curve is an \emph{abnormal curve} if there exist a co-vector $\lambda\in T^*_{\gamma(1)}M$ such that (\ref{e abnormale}) holds.
\label{d extremals}
\end{defi}

We can define normal curves also with domain different than $[0,1]$. 
\begin{defi}
Let $M$ be a subFinsler manifold with distribution $\Delta$ and let $I$ be an interval. A curve $\gamma:I\rr M$ is a \emph{normal curve} if the restriction of $\gamma$ to every compact sub-interval of $I$, when affinely re-parametrized in $[0,1]$, is a normal curve in the sense of Definition~\ref{d extremals}.
\label{d nextremals}
\end{defi}
The following remark ensures that the restriction of a normal curve to every sub-interval of its domain is a normal curve and therefore that Definition~\ref{d nextremals} is well posed.
\begin{oss}
Let $M$ be a subFinsler manifold with distribution $\Delta$.
Let $\gamma:[0,1]\rr M$ be a normal curve and $\lambda\in T^*_{\gamma(1)}M$ be a co-vector such that (\ref{e normale}) holds. For every $\alpha\in \mathbb{R}_{>0}$, define the curve $\gamma_\alpha:\left[0,\frac{1}{\alpha}\right]\rr M$ as
\begin{equation}
    \gamma_{\alpha}(t):=\gamma(\alpha t), \ \ \forall t\in\left[0,\frac{1}{\alpha}\right].
\end{equation}
Then
\begin{equation}
      (\phi^{\frac{1}{\alpha}}_{t;\alpha})^*\lambda_\alpha \in \partial_{\dot{\gamma}_\alpha(t)}\E_{\gamma_\alpha(t)}, \ \ \forall t\in \left[0,\frac{1}{\alpha}\right],
      \label{e reparametrizing}
\end{equation} 
where $\lambda_{\alpha}\in T^*_{\gamma(1)}M$is defined setting 
$\lambda_{\alpha}:=\alpha\lambda$ and $(\phi^{\frac{1}{\alpha}}_{t;\alpha})$ is the flow along the control $\alpha u(\alpha t) \mathbbm{1}_{\left[0,\frac{1}{\alpha}\right]} $ from $t$ to $\frac{1}{\alpha}$.
\label{o reparametrizing}
\end{oss}
\begin{proof}
By definition of sub-differential and of $\E$, if for some $p\in M$, $v\in T_pM$ and $a\in T_p^*M$ we have $a\in \partial_v\E_p$, then $\tau a\in  \partial_{\tau v}\E_p$ for every $\tau>0$.
Moreover, we have for all $t\in\left[0,\frac{1}{\alpha}\right]$ that $\gamma_\alpha'(t)=\alpha\gamma'(\alpha t)$ and that $\phi^{\frac{1}{\alpha}}_{t;\alpha}$ is equal to the flow $\phi_{\alpha t}^{1}$ along $u$ from $\alpha t$ to $1$. Consequently, since $(\phi_{\alpha t}^{1})^*\lambda \in \partial_{\dot{\gamma}(\alpha t)}\E_{\gamma(\alpha t)}$ for all $t\in\left[0,\frac{1}{\alpha}\right]$, we have that $(\phi^{\frac{1}{\alpha}}_{t;\alpha})^*(\alpha\lambda) \in \partial_{\dot{\gamma}_\alpha(t)}\E_{\gamma_\alpha(t)}$ for all $t\in\left[0,\frac{1}{\alpha}\right]$.
\end{proof}

From Proposition~\ref{p puntuale} and the formula of the differential of $\End$ in Proposition~\ref{p differentialend} we can get the equivalent of \eqref{e 2bis} for subFinsler manifolds.
\begin{cor}\emph{(Normal equation and end-point map)}
Let $M$ be a subFinsler manifold with distribution $\Delta$. Choose $X_1,...,X_m\in\Vecc(M)$ such that $\Span(\{X_1(p),...,X_m(p)\})=\Delta_p$ for every $p\in M$ and fix $q\in M$. Set $\Omega:=L^2([0,1],\mathbb{R}^m)$ and denote with $\End:\Omega\rr M$ the end-point map associated to $X_1,...,X_m$ and $q$. 
Let $\gamma:[0,1]\rr M$ be a normal curve with control $u$ such that $\gamma(0)=q$. Let $\lambda\in T^*_{\gamma(1)}M$ be a co-vector for which \eqref{e normale} holds. Then
\begin{equation}
    \lambda(\df\End_u u)=\int_0^1||\dot{\gamma}(t)||_{\gamma(t)}^2\df t,
    \label{e beingextremal}
\end{equation}

\label{p extremals}
\end{cor}
\begin{proof}
By Proposition~\ref{p differentialend} we have
\begin{equation}
    \lambda(\df\End_uv)=\int_0^1\lambda\left(\left(\phi_t^1\right)_*\sum_{i=1}^m v_i(t)X_i(\gamma(t))\right)\mathrm{d}t, \ \ \forall v\in \Omega.
    \label{e 2puntuale}
\end{equation}
By \eqref{e normale} we have
\begin{equation}
    \left(\left(\phi_t^1\right)^*\lambda\right)(\dot{\gamma}(t))=||\dot{\gamma}(t)||^2, \text{ for a.e. } t\in[0,1], 
    \label{e sub-differential_applied_to_derivative}
\end{equation}
see for example \cite[Lemma 2.19]{Hakavuori-2020-step2_geodesics}. Equations \eqref{e 2puntuale} and \eqref{e sub-differential_applied_to_derivative} imply \eqref{e beingextremal}.
\end{proof}
In the case of Carnot groups we always choose a left-invariant frame as the set of vector fields spanning the distribution $\Delta$ at every point. Therefore, we can interpret controls as functions $u\in L^2([0,1],\Delta_{1})$ and associate to each absolutely continuous curve a unique control: the left translation at the origin of the derivative of the curve. The flow along a control coincides with the right-translation along the associated curve, hence we can rewrite Proposition~\ref{p puntuale} and Corollary \ref{p extremals} as follows:
\begin{prop}
Let $G$ be a subFinsler Carnot group equipped with first stratum $V_1$ and a left-invariant norm, let $\{X_1,...,X_m\}$ be a left-invariant frame of the first stratum. Let $\End$ be the end-point map associated to $X_1,...,X_m$ and $1_G$. Let $\gamma:[0,1]\rr G$ be a normal curve with control $u\in L^2([0,1],V_1)$ and such that $\gamma(0)=1_G$. Then, there exist a right-invariant co-vector $\lambda$ such that
\begin{equation}
    \lambda\circ \Ad_{\gamma(t)}\in \partial_{u(t)}\E_{1}, \text{ for a.e. } t\in[0,1],
    \label{e normalLie}
\end{equation}
and
\begin{equation}
    \lambda(\df\End_u u)=||u||_{L^2}^2.
    \label{e beingextremalLie1}
\end{equation}
\label{p extremalsLie}
\end{prop}
\begin{proof}
Let $\Bar{\lambda}\in T^*_{\gamma(1)}G$ be a co-vector for which \eqref{e normale} holds. Define $\lambda$ to be the right-invariant extension of $\Bar{\lambda}$. 
Since the flow along $u$ is $\phi^t=R_{\gamma(t)}$, see \eqref{e flow_is_right_translation}, there holds
\begin{equation*}
   \left(\phi^1_t\right)^*\Bar{\lambda}=\Bar{\lambda}\circ\df R_{\gamma(1)}\circ\df R_{\gamma(t)}^{-1}, \forall t\in[0,1].
   \label{e pull-back_in_Lie}
\end{equation*}
Consequently, equation \eqref{e normale} rewrites as
\begin{equation}
    \Bar{\lambda}\circ\df R_{\gamma(1)}\circ\df R_{\gamma(t)}^{-1}\in\partial_{\dot{\gamma}(t)}\E_{\gamma(t)}, \text{ for a.e. } t \in[0,1]. 
    \label{e normale_Lie-proof1}
\end{equation}
Being $\dot{\gamma}=\df L_{\gamma}u$ (see \eqref{e curve_associated_to_control_Carnot}) and being $\E$ left-invariant, equation \eqref{e normale_Lie-proof1} is equivalent to
\begin{equation*}
   \Bar{\lambda}\circ\df R_{\gamma(1)}\circ\df R_{\gamma(t)}^{-1}\circ\df L_{\gamma(t)}\in\partial_{u(t)}\E_{1}, \text{ for a.e. } t \in[0,1],
    \label{e normale_Lie-proof2}
\end{equation*}
which is exactly \eqref{e normalLie} being $\lambda$ the right-invariant extension of $\Bar{\lambda}$.
Equation \eqref{e beingextremalLie1} follows from \eqref{e beingextremal}, being $\Bar{\lambda}=\lambda_{\gamma(1)}$ and $||u||_{L^2}=\int_0^1||\dot{\gamma}(t)||_{\gamma(t)}^2\df t$.
\end{proof}
From Proposition~\ref{p extremalsLie} and Remark~\ref{o reparametrizing} we have that if $G$ is a Carnot group, $I\subseteq\mathbb{R}$ is an interval, $\gamma:I\rr G$ is a normal geodesic and $u:=\df L_{\gamma}^{-1}\dot{\gamma}$, then there exists a right-invariant co-vector $\lambda$ such that $\lambda\circ \Ad_{\gamma(t)}\in \partial_{u(t)}\E_{1}$ for every $t\in I$.
\begin{defi}
Let $G$ be a Carnot group, $I$ be an interval and $\gamma:I\rr G$ be a normal curve. Set $u:=\df L_{\gamma}^{-1}\dot{\gamma}$. We say that $\lambda\in T^*G$ is a \emph{co-vector associated to} $\gamma$ if it is right-invariant and \begin{equation}
        \lambda\circ \Ad_{\gamma(t)}\in \partial_{u(t)}\E_{1}, \text{ for a.e. } t\in I.
    \label{e normalLie2}
\end{equation}
\label{d right-inv-covector}
\end{defi}

\section{Proof of main results \label{s mainteo}}
\subsection{The differential of the dilations}
This sub-section is devoted to a computation of the differential of the one-parameter family $\delta$ of dilations in Carnot groups. The aim is to write the vector field 
\begin{equation}
\Vec{\delta}(g):=\frac{\df}{\df\tau}\delta_\tau(g)\bigg|_{\tau=1}
\label{e defdelta}
\end{equation}
as a linear combination of right-invariant vector-fields. Then, using the particular form of the coefficients of this combination, we will get an estimate of $\lambda(\Vec{\delta}(g))$ in terms of the distance $d(1,g)$ for every right-invariant co-vector $\lambda$.
\begin{defi}
Let $G$ be a Carnot group with stratified Lie algebra $\mathfrak{g}=V_1\oplus...\oplus V_s$. An ordered basis $X_1,...,X_n$ of $\mathfrak{g}$ is \emph{adapted to the stratification} if 
\begin{equation*}
    X_i\in V_j, \ \ \forall j\in{1,...,s}, \forall i\in\{1,...,n\} \text{ s.t. } \dim(V_{j-1})< i\leq \dim(V_j),
\end{equation*}
where we are setting $V_0:=\{0\}$.\\
We say that a vector $X\in \mathfrak{g}$ has \emph{degree of homogeneity} $j\in\{1,...,s\}$ if $X\in V_j$. 
%
%
%
%
%
%
We say that a function $f:G\rr \mathbb{R}$  is \emph{homogeneous} of degree $\alpha\in \mathbb{R}$ if $f\circ\delta_\tau=\tau^\alpha f$ for all $\tau\geq 0$. 
\end{defi}
   
\begin{lemma}
 Let $G$ be a Carnot group, fix a basis $X_1,...,X_m$  adapted to the stratification and let $\Vec{\delta}$ be as in (\ref{e defdelta}). Then
 \begin{equation}
 \Vec{\delta}=\sum_{i=1}^m P_iX_i^\dagger,
 \label{e differenzialedelta}    
 \end{equation}
 where, for each $i\in \{1,...,m\}$,
 we denote by $ X_i^\dagger $ the right-invariant extension of $X_i$ and $P_i:G\rr \mathbb{R}$ are homogeneous functions of degree $d_i$,
 the integer $d_i\in \mathbb{N}$ being the degree of homogeneity of $X_i$.
 \label{l differenziale delta}
\end{lemma}
 \begin{proof}

Fix $\tau>0$. For all $g\in G$ we have
\begin{equation}
    \left(\left(\delta_\tau\right)_*\Vec{\delta}\right)(g)=\frac{\df}{\df \epsilon}\delta_\tau\circ\delta_{1+\epsilon}(\delta_{\tau}^{-1}(g)) \bigg|_{\epsilon=0}=\frac{\df}{\df \epsilon}\delta_{1+\epsilon}(g) \bigg|_{\epsilon=0}=\Vec{\delta}(g).
    \label{e pushforwarddelta}
\end{equation}
Moreover, for all $i\in\{1,...,m\}$ there holds
\begin{equation}
    \left(\delta_\tau\right)_*X_i^{\dagger}=\tau^{d_i}X_i^{\dagger},
    \label{e pushforwardXi}
\end{equation}
since $\left(\delta_\tau\right)_*X_i^{\dagger}(1_G)=\tau^{d_i}X_i^{\dagger}(1_G)$ and $\left(\delta_\tau\right)_*X_i^{\dagger}$ is a right-invariant vector field, being $\delta_\tau$ a group homomorphism and $X_i^\dagger$ right-invariant.\\
Since $\Vec{\delta}$ is a smooth vector field there exist smooth functions $P_1,...,P_m:G\rr \mathbb{R}$ such that \eqref{e differenzialedelta} holds. To conclude the proof of the lemma we have to show that 
 \begin{equation}
     \label{e Pibeinghomogeneous}
      P_i\circ\delta_\tau=\tau^{d_i}P_i, \ \  \forall i\in \{1,...,m\}.
 \end{equation}
For all $i\in\{1,...,m\}$ we have
\begin{eqnarray*}
    \sum_{i=1}^m P_iX_i^{\dagger}&\stackrel{\eqref{e differenzialedelta}}{=}&\Vec{\delta} \\
    &\stackrel{\eqref{e pushforwarddelta}}{=}& \left(\delta_\tau\right)_*\Vec{\delta} \\
    &\stackrel{\eqref{e differenzialedelta}}{=}& \left(\delta_\tau\right)_*\left(\sum_{i=1}^m P_iX_i^{\dagger}\right)\\
    &\stackrel{\eqref{e pushforwardXi}}{=}& \sum_{i=1}^m \tau^{d_i}(P_i\circ \delta_{\tau}^{-1})X_i^{\dagger},
\end{eqnarray*}
and consequently \eqref{e Pibeinghomogeneous} holds.
 \end{proof} 
The identity found in Lemma~\ref{l differenziale delta} allows us to give an estimate of the value that we get applying a right-invariant co-vector to the vector $\Vec{\delta}(g)$ defined in (\ref{e defdelta}). Up to constants, the bound we find depends only on the co-vector and on the distance of $g$ from the origin. To get a quantitative estimate we fix a norm on the Lie algebra. 
\begin{defi}
Let $G$ be a Lie group and $\lambda$ a right-invariant co-vector. Choose a basis $\{X_1,...,X_n\}$ of $\,T_1G$ and define
\begin{equation}
    N(\lambda):=\sum_{i=1}^n|\lambda(X_i)|.
    \label{e normalambda}
\end{equation}
Up to a multiplicative constant
the norm $N$ is equivalent to every other norm on the space of right-invariant co-vectors.
\label{d normalambda}
\label{d Nlambda}
\end{defi}
\begin{lemma}
Let $G$ be a Carnot group. Define $\Vec{\delta}$ as in (\ref{e defdelta}). There exists a constant $C>0$ such that for every right-invariant co-vector $\lambda$ we have
\begin{equation}
    \lambda(\Vec{\delta}(g))\leq CN(\lambda)\max(d(1,g),d(1,g)^s), \qquad\forall g\in G,
    \label{e stimalambdasudelta}
\end{equation}
where $N(\lambda)$ is defined as in Definition~\ref{d normalambda}.
\label{l stimalambdasudelta}
\end{lemma}
\begin{proof}
Choose a basis $X_1,...,X_n$ adapted to the stratification; call $\{X_i^\dagger\}_{i\in\{1,...,n\}}$ the right-invariant extension of this basis. Let $d_i\in\mathbb{N}$ be the degree of homogeneity of $X_i$. From Lemma~\ref{l differenziale delta} we have that \eqref{e differenzialedelta} holds. For all $i\in\{1,...,n\}$, being the function $P_i$ homogeneous of degree $d_i$, there exist a constant $C_i>0$ such that
\begin{equation}
   P_i(g)\leq C_i d(1,g)^{d_i}, \ \ \forall g \in G.
   \label{e inequalitihomogeneity}
\end{equation}
Setting $C=\max_{1,...,n}C_i$, for every right-invariant co-vector $\lambda$, by definition of $N(\lambda)$, we have
\begin{equation}
    \lambda(\vec{\delta}(g))\stackrel{\eqref{e differenzialedelta}}{=}\sum_{i=1}^n  P_i(g)\lambda(X_i)\stackrel{\eqref{e normalambda}, \eqref{e inequalitihomogeneity}}{\leq} CN(\lambda)\max(d(1,g),d(1,g)^s), \ \ \forall g\in G.
    \label{e stimadelta2}
\end{equation}
This concludes the proof of the lemma.
\end{proof}
\subsection{Proof of Theorem~\ref{t fugacptquant}}
We are now ready to prove the main result of this paper.
\begin{proof}[Proof of Theorem~\ref{t fugacptquant}]
Let $\gamma:\mathbb{R}\rr G$ be a normal curve parametrized by arc-length. Without loss of generality we can assume $t'=0$, $t>0$ and $\gamma(0)=1_G$.
Since all norms on the space of right-invariant co-vectors are equivalent, we don't lose generality assuming that $N$ is the norm defined by \eqref{e normalambda}.
Choose a 
co-vector $\lambda$ associated to $\gamma$ 
(see Definition~\ref{d right-inv-covector}). Denote with $V_1$ the first stratum of the stratification of $\Lie(G)$ and for $t>0$ denote with $u_t\in L^2([0,1],V_1)$ the control of the curve $\gamma_t:[0,1]\rr G$, $\gamma_t(\tau):=\gamma|_{[0,t]}(t\tau)$ for all $\tau\in[0,1]$.
For every $t>0$, we have
\begin{equation}
    \df\End_{u_t} u_t=\frac{\df}{\df\tau}\End(\tau u_t)\bigg|_{\tau=1}.
    \label{e 1.3mainteo}
\end{equation}
Being $\tau u_t$ the control of $\delta_\tau \gamma_t $,
we can rewrite (\ref{e 1.3mainteo}) as 
\begin{equation}
     \df \End_{u_t} u_t= \frac{\df}{\df\tau}\delta_\tau(\End(u_t))\bigg|_{\tau=1}=\Vec{\delta}(\End(u_t)), \ \  \forall t>0,
     \label{e 1.6mainteo}
\end{equation}
where $\vec{\delta}$ is defined by \eqref{e defdelta}.
Applying Lemma~\ref{l stimalambdasudelta} we have that there exists a constant $C>0$ such that
\begin{equation}
    \lambda(\Vec{\delta}(\End(u_t)))\leq C 
    N(\lambda)\max(d(1,\End(u_t)),d(1,\End(u_t))^s), \ \ \forall t>0.
    \label{e 2mainteo}
\end{equation}
By \eqref{e 1.6mainteo} and \eqref{e 2mainteo}, and being $\End(u_t)=\gamma(t)$ for all $t>0$, we have
\begin{equation}
    \lambda(\df \End_{u_t} u_t)\leq C
    N(\lambda)\max(d(1,\gamma(t)),d(1,\gamma(t))^s), \ \ \forall t>0.
    \label{e 3mainteo}
\end{equation}
Moreover, we have from Remark~\ref{o reparametrizing} and Proposition~\ref{p extremalsLie} that
\begin{equation}
    t\lambda(\df \End_{u_t} u_t)=\int_0^1||t\dot{\gamma}(tx)||^2\df x=t^2, \ \ \forall t>0. 
    \label{e 1mainteo}
\end{equation}
From equations (\ref{e 3mainteo}) and (\ref{e 1mainteo})  we get for every $t>0$ that
\begin{equation}
    t\leq C
    N(\lambda)\max(d(1,\gamma(t)),d(1,\gamma(t))^s).
    \label{e 4mainteo}
\end{equation}
In particular, for $t>C
N(\lambda)$ we must have \begin{equation*}
    d(1,\gamma(t))>1,
\end{equation*}
and therefore
\begin{equation*}
    \max(d(1,\gamma(t))),d(1,\gamma(t))^s)=d(1,\gamma(t))^s.
\end{equation*}
Thus, for all $t>C
N(\lambda)$, Equation (\ref{e 4mainteo}) becomes 
\begin{equation}
  t\leq C
  N(\lambda)d(1,\gamma(t))^s.
  \label{e 5mainteo}
\end{equation}
Set $\epsilon:=C^{-\frac{1}{s}}$. 
If $t>C
N(\lambda)$ equation \eqref{e notinball} is a conseqeuence of \eqref{e 5mainteo}. If $t\leq C
N(\lambda)$ equation \eqref{e notinball} is trivially true being the right-hand side less or equal than $0$. We showed that \eqref{e notinball} holds for every $t>0$, thus we concluded the proof of the theorem.
\end{proof}

\subsection{Proof of Corollary~\ref{c noperiodic0}}
This section is devoted to the proof of Corollary~\ref{c noperiodic0}. We will actually prove a stronger statement:
\begin{teo}
In self-similar spaces, normal loops starting from the origin are constant. 
\label{c noperiodic}
\end{teo}
We start showing that normal curves in self-similar spaces can be lifted to normal curves in Carnot groups.
\begin{lemma}
Let $G$ be a Carnot group with stratified Lie algebra $\mathfrak{g}=V_1\oplus...\oplus V_s$ and let $H<G$ be a dilation-invariant subgroup, with $\Lie(H)\cap V_1=\{0\}$. Denote with $\pi:G\rr \mfaktor{H}{G}$ the projection. Let $\gamma:[0,1]\rr \mfaktor{H}{G}$ be a normal curve. For every $g\in \pi^{-1}(\gamma(0))$ there exists a normal curve $\Tilde{\gamma}:[0,1]\rr G$ such that $\Tilde{\gamma}(0)=g$ and $\pi\circ \Tilde{\gamma}=\gamma$. Moreover, we can choose a co-vector $\Tilde{\lambda}\in (T_{\Tilde{\gamma}(1)}G)^*$  associated to $\Tilde{\gamma}$ such that
\begin{equation}
    \Tilde{\lambda}(\ker (\df\pi_{\Tilde{\gamma}(1)}))=\{0\}.
    \label{e tildelambdaonh}
\end{equation}
\label{l lift}
\end{lemma}
\begin{proof}
Denote with $\Delta$ the distribution of $G$ and with $\Bar{\Delta}$ the one of $\mfaktor{H}{G}$, see Definition~\ref{def_selfsimilar}
. Fix a left-invariant frame $\{X_1,...,X_m\}$ such that $\Span\{X_1(g),...,X_m(g)\}=\Delta_g$ for all $g\in G$. We have $\Span\{\df\pi_g X_1(g),...,\df\pi_g X_m(g)\}=\df\pi_g \Delta_g =\Bar{\Delta}_{\pi(g)} $ for all $g\in G$, thus we can see $\Omega:=L^2([0,1],\mathbb{R}^m)$ both as the set of controls of curves in $G$ associated to $\{X_1,...,X_m\}$ and as the set of controls of curves in $\mfaktor{H}{G}$ associated to $\{\pi_*X_1,...,\pi_*X_m\}$.
%
%
%
%
%
We remark that by definition of $||\cdot||_{\mfaktor{H}{G}}$ in (\ref{d defnormadown}), and being the norm on $\Delta$ left-invariant, 
we have for every $p\in \mfaktor{H}{G}$, every $q \in \pi^{-1}(p)$ and every $v\in\mathbb{R}^m$ that \begin{equation*}
  ||\sum_{i=1}^m v_i\left(\pi_*X_i\right)(p) ||_{\mfaktor{H}{G}}  =||\sum_{i=1}^m v_iX_i(q) ||_G
\end{equation*}
Thus for the control $u\in \Omega$ of $\gamma$ we have
\begin{equation}
    ||\dot{\gamma}(t)||_{\mfaktor{H}{G}}=||\sum_{i=1}^m u_i(t)X_i(q) ||_G, \text{ for a.e. } t\in[0,1], \forall q\in \pi^{-1}(\gamma(t)).
    \label{e choiceofu}
\end{equation}
To simplify the notation we write $u(t)\cdot X(q)$ for $\sum_{i=1}^m u_i(t)X_i(q)$.
We claim that for a.e. $t\in [0,1]$, for all $q\in \pi^{-1}(\gamma(t))$ and for all linear functions $a:T_{\gamma(t)}(\mfaktor{H}{G})\rr \mathbb{R}$ there holds
\begin{equation}
    a\in \partial_{\dot{\gamma}(t)}\left(\frac{\left(||\cdot||_{\mfaktor{H}{G}}\right)_{\gamma(t)}^2}{2}\right)\implies a\circ \df\pi_q \in \partial_{u(t)\cdot X(q)}\left(\frac{\left(||\cdot||_G\right)_{q}^2}{2}\right).
    \label{e subdifferentialinclusion}
\end{equation}
Indeed, by definition of sub-differential if $a\in \partial_{\dot{\gamma}(t)}\left(\frac{\left(||\cdot||_{\mfaktor{H}{G}}\right)_{\gamma(t)}^2}{2}\right)$ and $q\in \pi^{-1}(\gamma(t))$, then
\begin{equation}
   \begin{split}
           \frac{||\dot{\gamma}(t)||_{\mfaktor{H}{G}}^2}{2}
           \leq \frac{||\df\pi_q(w)||^2_{\mfaktor{H}{G}}}{2}-a(\df\pi_q(w)-\dot{\gamma}(t)), \forall w\in T_qG.
   \end{split}
 \end{equation}
 Consequently, by (\ref{e choiceofu}) we have
 \begin{equation}
     \frac{||u(t)\cdot X(q)||_G^2}{2}\leq \frac{||\df\pi_q(w)||^2_{\mfaktor{H}{G}}}{2}-a(\df\pi_q(w)-\dot{\gamma}(t)), \forall w\in T_qG.
     \label{e bo}
 \end{equation}
 By definition of $||\cdot||_{\mfaktor{H}{G}}$ we have $||\df\pi_q(w)||^2_{\mfaktor{H}{G}}\leq  ||w||_G^2$ and therefore Equation (\ref{e bo}) implies
 \begin{equation}
      \frac{||u(t)\cdot X(q)||_G^2}{2}\leq \frac{||w||_G^2}{2}-(a\circ\df\pi_q)(w-u(t)\cdot X(q)), \forall w\in T_qG,
 \end{equation}
 which proves the claim (\ref{e subdifferentialinclusion}).\\
If we call $\Tilde{\phi}^t$ the flow along $u$ in $G$ and $\phi^t$ the flow along $u$ in $\mfaktor{H}{G}$, we have that 
\begin{equation}
    \pi\circ \Tilde{\phi}^t=\phi^t\circ \pi, \ \ \forall t\in[0,1],
    \label{e flowcommuteswithprojection}
\end{equation}
since for every $q\in G$ both $\pi\circ \Tilde{\phi}^t(q)$ and $\phi^t\circ \pi(q)$ solve the differential equation
\begin{equation*}
    \begin{cases}
    \sigma'(t)=\sum_{i=1}^m u_i(t)\pi_*X_i(\sigma(t));\\
    \sigma(0)=\pi(q).
    \end{cases}
\end{equation*}
In particular, for all $g\in \pi^{-1}(\gamma(0))$, the curve $\Tilde{\gamma}:[0,1]\rr G$, defined as $\Tilde{\gamma}(t):=\Tilde{\phi}^t(g)$ for all $t\in[0,1]$, is such that $\pi\circ \Tilde{\gamma}=\gamma$ and $\Tilde{\gamma}(0)=g$.\\
Being $\gamma$ normal there exists $\lambda\in T^*_{\gamma(1)}(\mfaktor{H}{G})$ such that
\begin{equation}
   \left(\phi_t^1\right)^* \lambda \in \partial_{\dot{\gamma}(t)}\left(\frac{\left(||\cdot||_{\mfaktor{H}{G}}\right)_{\gamma(t)}^2}{2} \right), \text{ for a.e. } t\in [0,1].
    \label{e normalforgamma}
\end{equation}
We prove that $\Tilde{\gamma}$ is normal showing that for 
$\Tilde{\lambda}:=\lambda\circ \df\pi_{\Tilde{\gamma}(1)}$ there holds
\begin{equation}
    \left(\Tilde{\phi}_t^1\right)^*\Tilde{\lambda}\in \partial_{\dot{\Tilde{\gamma}}(t)}\left(\frac{\left(||\cdot||_G\right)_{\Tilde{\gamma}(t)}^2}{2}\right), \text{ for a.e. } t\in [0,1].
\end{equation}
%
Indeed, we know from (\ref{e subdifferentialinclusion}) and \eqref{e normalforgamma} that
\begin{equation}
  \left(\left(\phi_t^1\right)^*\lambda\right)\circ \pi_* \in \partial_{\dot{\Tilde{\gamma}}(t)}\left(\frac{\left(||\cdot||_G\right)_{\Tilde{\gamma}(t)}^2}{2}\right), \text{ for a.e. } t\in [0,1],
  \label{e mainteo-qf}
\end{equation}
and by (\ref{e flowcommuteswithprojection}) we have
\begin{equation}
   \left(\Tilde{\phi}_t^1\right)^*\Tilde{\lambda}= \left(\left(\phi_t^1\right)^*\lambda\right)\circ \pi_*, \ \ \forall t\in[0,1].
\end{equation}
Finally, Equation (\ref{e tildelambdaonh}) come trivially from the definition of $\Tilde{\lambda}$.
\end{proof}

We are now ready to prove that in self-similar spaces normal loops starting from the origin are constant.
\begin{proof}[Proof of Theorem~\ref{c noperiodic}]
Let $G$ be a Carnot group and $H<G$ be a subgroup invariant under the action of dilations, with $\Lie(H)\cap V_1=\{0\}$. Denote with $\pi:G\rr \mfaktor{H}{G}$ the canonical projection. Let $\gamma:[0,1]\rr \mfaktor{H}{G}$ be a normal curve such that $\gamma(0)=\gamma(1)=\pi(1_G)$. From Lemma~\ref{l lift} we know there exists exists a normal curve $\Tilde{\gamma}:[0,1]\rr G$ with control $u$ such that $\pi\circ\Tilde{\gamma}=\gamma$ and $\Tilde{\gamma}(0)=1_G$. In particular $\Tilde{\gamma}(1)\in H$ and thus  \begin{equation}
\label{formula_end_d}
\df\End_u u=\frac{\df}{\df \tau}\delta_\tau(\Tilde{\gamma}(1)) \bigg|_{\tau=1}\in T_{\Tilde{\gamma}(1)}H,
\end{equation}
being $H$ dilation invariant. On the other hand, by (\ref{e tildelambdaonh}) we know that $\Tilde{\gamma}$ solves Equation (\ref{e beingextremalLie1}) with respect to a covector $\Tilde{\lambda}\in T^*_{\Tilde{\gamma}(1)}G$ such that $\Tilde{\lambda}(T_{\Tilde{\gamma}(1)}H)=\{0\}$, thus
\begin{equation}
    ||u||_{L^2}=\Tilde{\lambda}(\df\End_u u)=0.
\end{equation}
We proved that $\Tilde{\gamma}$ is constant and thus $\gamma$ is constant. Since $\gamma$ was an arbitrary normal loop starting from the origin, we concluded the proof of the theorem.
\end{proof}




\section{Some examples}
\subsection{Curves with end-point arbitrarely close to the origin}
Even if by Theorem~\ref{c noperiodic} we know that in Carnot groups non-constant normal loops don't exists, we can find normal curves of length $1$ with end-point arbitrarily close to the origin. We present a simple example in the Heisemberg group.
\begin{defi}
The \emph{3-dimensional Heisemberg group} is the Carnot group $(\mathbb{R}^3,\cdot)$ with product law
\begin{equation*}
    (x,y,z)\cdot(x',y',z'):=(x+x',y+y',z+z'-\frac{1}{2}(x'y-xy')), \forall x,y,z,x',y',z'\in \mathbb{R}.
\end{equation*}
In the Heisemberg group we choose as orthonormal left-invariant frame of the first stratum the two vector fields $\{X,Y\}$ defined for all $x,y,z\in \mathbb{R}$ by
\begin{eqnarray*}
    X(x,y,z):=\partial_x-\frac{y}{2}\partial_z;\\ Y(x,y,z):=\partial_y+\frac{x}{2}\partial_z.
\end{eqnarray*}
\label{d Heisemberg_group}
\end{defi}
\begin{prop}
Let $G:=(\mathbb{R}^3,\cdot)$ be the subRiemannian 3-dimensional Heisemberg group. For every $\epsilon>0$ there exists a normal curve $\gamma:[0,1]\rr \mathbb{R}^3$, parametrized by arclength, such that $\gamma(0)=1_G$ and $d(\gamma(1),1_G)<\epsilon$.
\end{prop}
\begin{proof}
For every $N\in\mathbb{N}$, in the coordinates of Definition~\ref{d Heisemberg_group}, consider the curve $\gamma_N:[0,1]\rr G$ defined as $\gamma_N(t):=(x_N(t),y_N(t),z_N(t))$ with
\begin{equation*}
    x_N(t):=\frac{\cos(2\pi Nt)-1}{2\pi N};
\end{equation*}
\begin{equation*}
    y_N(t):=\frac{\sin(2\pi Nt)}{2\pi N};
\end{equation*}
\begin{equation*}
    z_N(t):=\frac{2\pi Nt-\sin(2\pi Nt)}{8(\pi N)^2}.
\end{equation*}
The curve $\gamma_N$ is the horizontal lift of a circle of radius $\frac{1}{2\pi N}$ travelled $N$ times and is a normal geodesic parametrized by arc-length (see for example \cite[Section 1.4.1]{ledonne2010lecture} for a characterization of geodesics in the Heisemberg group). We have $\gamma_N(0)=1_G$ and $\gamma_N(1)=(0,0,\frac{1}{4\pi N})$. Therefore, for every $\epsilon>0$ there exists $N>0$ such that $\gamma_N(1)\in B(1_G,\epsilon)$. We showed that for every $\epsilon>0$ there exists a normal curve of length $1$ having end-point in $B(1_G,\epsilon)$, thus the proof of the proposition is concluded.
\end{proof}

\subsection{Optimality of the exponent in Theorem~\ref{t fugacptquant}}
In an arbitrary Carnot group we cannot obtain an estimate as the one in Theorem~\ref{t fugacptquant} with an exponent smaller than the reciprocal of the step. Indeed, for every $s\in\mathbb{N}$, we provide an example of a normal curve $\gamma$, in a filiform group of step $s$, for which the distance of $\gamma(t)$ from the origin is bounded above by a constant times $t^{\frac{1}{s}}$.
\begin{defi}
The \emph{subRiemannian filiform group of first type} is the Carnot group with stratified Lie algebra $\mathfrak{g}=V_1\oplus...\oplus V_s$ with basis $X_1,Y_1,...,Y_s$, 
$\{X_1,Y_1\}$ being an orthonormal basis of $V_1$, $V_j=\Span(Y_j)$ for $j=2,...,s$ and only non-trivial bracket relations $[X_1,Y_i]=Y_{i+1}$, for every $i=1,...,s-1$. 
\end{defi}
We recall that in the subRiemannian filiform groups of first type all length-minimizing curves are normal (see for example \cite[Proposition 4.1]{bryant1993rigidity}).

\begin{prop}
Let $G$ be the sub-Riemmanian filiform group of first type of step $s$. Let $Y_n$ be a vector spanning the $s$-th stratum and $\gamma:\mathbb{R}\rr G$ be a normal curve with $\gamma(0)=1_G$, $\gamma(1)=\exp(Y_n)$ and $\gamma|_{[0,1]}$ length-minimizing. Then there exists $C>0$ such that for every $t>1$ there holds
\begin{equation}
d(1,\gamma(t)) <Ct^{\frac{1}{s}}.   
\end{equation}
\label{p esempio filiforme}
\end{prop}
\begin{proof}
Let $\{X_1,X_2\}$ be an orthonormal frame of the first stratum. Let $\lambda$ be a 
co-vector associated to $\gamma$, see Definition \ref{d right-inv-covector}. Since the energy at the origin is differentiable with differential
\begin{equation*}
    \df(\E_{1})_w v=\langle w,v\rangle, \ \ \forall v,w\in T_1G,
\end{equation*}
equation (\ref{e normalLie2}) rewrites as
\begin{equation}
    \lambda(\Ad_{\gamma(t)}X_i)=\langle u(t),X_i\rangle, \ \ \forall t\in \mathbb{R}, \forall i\in\{1,2\}.
    \label{e 2es2}
\end{equation}
We claim that
\begin{equation}
    \gamma(m+s)=\exp(mY_n)\gamma(s), \ \ \forall m\in \mathbb{N}, \forall s\in\mathbb{R}.
    \label{e char56}
\end{equation}
Indeed, the curve $\exp(mY_n)\gamma(s) $ solves the normal Equation (\ref{e 2es2}) with respect to the covector $\lambda$:
\begin{equation}
   \begin{split}
        (\exp(mY_n)\gamma(s))'=\df L_{\exp(mY_n)}\gamma(s)'=\\ =\df L_{\exp(mY_n)}\sum_{i=1}^2\lambda(\Ad_{\gamma(s)}X_i)X_i(\gamma(s))=\\ =\sum_{i=1}^2\lambda(\Ad_{\exp(mY_n)\gamma(s)}X_i)X_i(\exp(mY_n)\gamma(s)),  
   \end{split}
\end{equation}
where in the second equality we used that $\gamma$ solves (\ref{e 2es2}) and the last equality comes from the fact that $\exp(mY_n)$ is in the center.\\
We have
\begin{equation}
d(1,\exp(mY_n))=d(1,\delta_{m^{\frac{1}{s}}}(\exp(Y_n))) =m^{\frac{1}{s}}d(1,\exp(Y_n)), \ \ \forall m\in \mathbb{N}.  
\label{e fromballbox}
\end{equation}
Equation (\ref{e fromballbox}) together with (\ref{e char56}) and the triangle inequality give us
\begin{equation}
    d(1,\gamma(m+s)) \leq m^{\frac{1}{s}}d(1,\exp(Y_n))+d(1,\gamma(s)), \ \ \forall m\in\mathbb{N}, \forall s\in \mathbb{R}.
\end{equation}
In particular, being $\gamma|_{[0,1]}$ length minimizing and $\gamma(1)=\exp(Y_n)$, we have
\begin{equation}
   \begin{split}
   d(1,\gamma(m+s)) \leq m^{\frac{1}{s}}d(1,\exp(Y_n))+d(1,\exp(Y_n))\leq\\ \leq d(1,\exp(Y_n))(m^{\frac{1}{s}}+1), \ \ \forall m\in \mathbb{N}, \forall s\in [0,1].
   \end{split}
   \label{e 3es2}
\end{equation}
Consequently, setting $C=2d(1,\exp(Y_n))$, we have for $m>1$ that
\begin{equation}
   d(1,\gamma(m+s)) <Cm^{\frac{1}{s}}<C(m+s)^{\frac{1}{s}}, \ \ \forall m\in\mathbb{N}, \forall s \in[0,1].
\end{equation}
This concludes the proof of the proposition.
\end{proof}

\bibliographystyle{plain}
\bibliography{bibliografia.bib}

\end{document}